\documentclass[]{amsproc}

\usepackage{lineno,hyperref}
\usepackage{amsmath,amssymb,amscd,euscript}
\usepackage{dsfont}
\usepackage{subfigure}
\usepackage{graphicx}
\modulolinenumbers[5]

\newtheorem{definition}{Definition}[section]
\newtheorem{theorem}{Theorem}[section]
\newtheorem{lemma}[theorem]{Lemma}
\newtheorem{remark}[theorem]{Remark}
\newtheorem{corollary}[theorem]{Corollary}
\newtheorem{example}[theorem]{Example}

\allowdisplaybreaks \allowdisplaybreaks[4]

\begin{document}

\title[]{Strong convergence rate of the Euler scheme for SDEs driven by additive rough fractional noises}

\author{Chuying Huang}
\address{School of Mathematics and Statistics \& FJKLMAA, Fujian Normal University, Fuzhou 350117, China}
\curraddr{}
\email{huangchuying@fjnu.edu.cn; huangchuying@lsec.cc.ac.cn (Corresponding author)}
\author{Xu Wang}
\address{Academy of Mathematics and Systems Science, Chinese Academy of Sciences, Beijing 100190, China; School of Mathematical Sciences, University of Chinese Academy of Sciences, Beijing 100049, China}
\curraddr{}
\email{wangxu@lsec.cc.ac.cn}


\date{\today}

\dedicatory{}

\begin{abstract}
The strong convergence rate of the Euler scheme for SDEs driven by additive fractional Brownian motions is studied, where the fractional Brownian motion has Hurst parameter $H\in(\frac13,\frac12)$ and the drift coefficient is not required to be bounded. The Malliavin calculus, the rough path theory and the $2$D Young integral are utilized to overcome the difficulties caused by the low regularity of the fractional Brownian motion and the unboundedness of the drift coefficient. The Euler scheme is proved to have strong order $2H$ for the case that the drift coefficient has bounded derivatives up to order three and have strong order $H+\frac12$ for linear cases. Numerical simulations are presented to support the theoretical results.
\end{abstract}

\maketitle





\section{Introduction}

Stochastic differential equations driven by fractional Brownian motions with Hurst parameter $H\in(0,1)$ are basic models to characterize the randomness phenomena and have various applications in the fields of hydrology \cite{fbm}, porous media \cite{liu17SINUM}, oscillators \cite{HHW18}, explorations \cite{FLW20}, finance \cite{hhkw} and so on. If $H>\frac12$, the fractional Brownian motion (fBm) exhibits a long-range dependence property.
If $H=\frac12$, the fBm is equivalent to the standard Brownian motion so that the increments are independent. 
If $H<\frac12$, the fBm exhibits a short-range dependence property and the regularity of the sample paths is relatively low, in which case we call it rough fractional noise.
In this article, we investigate the numerical approximation  for the stochastic differential equation (SDE) driven by an additive  rough fractional noise
\begin{align}\label{sde}
	{\rm d}X_t=a(X_t){\rm d} t +\sigma {\rm d}B_t,\quad t\in(0,T],
\end{align}
where $X_0\in\mathbb{R}$ is a deterministic  initial value, the drift coefficient $a$ is unbounded, and $B=\{B_t\}_{t\in[0,T]}$ is the fBm with Hurst parameter $H\in(\frac13,\frac12)$.

One of the main obstacles in the convergence analysis on numerical schemes for SDEs in the rough case is the low regularity of the noise, which leads to the lack of an explicit formulation for the covariance kernel of the noise. Meanwhile, the unboundedness of the drift coefficient and the correlation of the increments of the fBm make the interaction of the local errors between the numerical solution and the exact solution more complicated. These difficulties result in that the numerical analysis in this case is far from well-developed. To deal with the problems mensioned above, we apply the Malliavin calculus, the rough path theory and the $2$D Young integral to establish the strong convergence rate of the Euler scheme for \eqref{sde}. More precisely, 
for $n\in\mathbb{N}_+$, denoting $h=\frac{T}{n}$ and $t_k=kh$, we focus on the following continuous interpolation of the Euler scheme
\begin{align}\label{sch}
	Y_{t}=Y_{t_k}+a(Y_{t_k})(t-t_k) +\sigma \big(B_t-B_{t_k}\big),\quad t\in(t_k,t_{k+1}],~k=0,\cdots,n-1.
\end{align}
Our main result is stated in the following.
\begin{theorem}\label{main}
	Let $H\in(\frac13,\frac12)$. Assume that $a:\mathbb{R}\rightarrow\mathbb{R}$ has  bounded derivatives up to order three. Then it holds that 
	\begin{align*}
		\bigg(\sup_{t\in[0,T]}\mathbb{E}\big|X_t-Y_t\big|^2\bigg)^{1/2}\le Ch^{2H},
	\end{align*}
	where $X$ solves \eqref{sde} and $Y$ is given by the the Euler scheme \eqref{sch}.
\end{theorem}

As $H$ tends to $\frac12$, the  strong convergence rate of the Euler scheme above goes to $1$, which is consistent with the classical result that the Euler--Maruyama scheme for SDEs driven by additive standard Brownian motions has strong order $1$ \cite[Chapter 1]{MT2004book}. Moreover, comparing with \cite{MLMC16,Harxiv,AAP2019}, Theorem \ref{main} reveals that  the strong convergence rate of the Euler scheme in  the above additive noise case is half order higher than those of the Euler-type schemes in the multiplicative noise case. In particular, if $a$ is linear, we have that the strong convergence rate of the Euler scheme is improved to $H+\frac12$; see Corollary \ref{cor}.  We remark that the results of Theorem \ref{main} and Corollary \ref{cor} can be extended directly to multi-dimensional cases.  
If the drift coefficient is bounded but less regular, we refer to \cite{BDG2021} for the optimal strong convergence rate of the Euler scheme in H\"older spaces.

The rest of the article is arranged as follows. In Section \ref{sec-2}, some preliminaries for the $2$D Young integral and the Malliavin calculus are introduced. In Section \ref{sec-3}, we prove the strong convergence rate of  the Euler scheme, i.e., Theorem \ref{main} and Corollary \ref{cor}. In Section \ref{sec-4}, numerical simulations are given to verify our theoretical analysis.

\section{$2$D Young integral and Malliavin calculus}\label{sec-2}
This section reviews basic concepts and results  about the $2$D Young integral and the Malliavin calculus associated to the fBm. 
We utilize $C$ as a generic constant and $G$ as a generic finite random variable, which may be different  from line to line. We will make use of  substripts to emphasize the parameters that they depend on.

\subsection{$2$D Young integral}
Let $U,W$ be Banach spaces with norms $\|\cdot\|_U$ and $\|\cdot\|_W$, respectively. We denote by $\mathcal{L}(U,W)$  the set of linear operators from $U$ to $W$. 

\begin{definition}
	For fixed $p\ge 1$ and $T>0$, the $p$-variation of $f:[0,T]\rightarrow U$ on $[s,t]\subseteq[0,T]$ is defined as 
	\begin{align*}
		\|f\|_{p\text{-}var;[s,t]}:=\sup_{\mathcal{P}\in\mathcal{D}([s,t])}\bigg(\sum_{k=0}^{N-1}\big\|f_{t_{k+1}}-f_{t_k}\big\|_U^p\bigg)^{1/p},
	\end{align*}
	where $\mathcal{P}=\{t_k:k=0,\cdots,N,~s=t_0<t_1<\cdots<t_N=t\}$ denotes a partition of $[s,t]$ and $\mathcal{D}([s,t])$ is the set of all such partitions. In addition, we define 
	\begin{align*}
		C^{p\text{-}var}(U;[0,T]):=\{f :\|f\|_{p\text{-}var;[0,T]}<+\infty\}.
	\end{align*}
\end{definition}

\begin{definition}
	Fix $p\ge 1$ and $T>0$. For $g:[0,T]^2\rightarrow U$, let
	\begin{align*}
		g([u_i,u_{i+1}]\times[v_j,v_{j+1}]):=g_{u_{i+1},v_{j+1}}-g_{u_{i+1},v_{j}}-g_{u_{i},v_{j+1}}+g_{u_{i},v_{j}}.
	\end{align*}
	The $p$-variation of $g$ on $[s,t]\times[u,v]\subseteq[0,T]^2$ is defined as 
	\begin{align*}
		\|g\|_{V^p;[s,t]\times[u,v]}:=\sup_{\pi\in\mathcal{D}([s,t]\times[u,v])}\bigg(\sum_{i,j}\big\|g([u_i,u_{i+1}]\times[v_j,v_{j+1}])\big\|_U^p\bigg)^{1/p},
	\end{align*}
	where $\pi=\{(u_i,v_j)\}$ is a partition of  $[s,t]\times[u,v]$ and $\mathcal{D}([s,t]\times[u,v])$ denotes the set of grid-like partitions of  $[s,t]\times[u,v]$. Moreover, we define
	\begin{align*}
		C^{p\text{-}var}(U;[0,T]^2):=\{g:\|g\|_{V^p;[0,T]^2}<+\infty\}.
	\end{align*}
\end{definition}

\begin{remark}\label{r1}
	For $f:[0,T]\rightarrow U$, the $\beta$-H\"older semi-norm of $f$ on $[s,t]\subseteq[0,T]$  is denoted by
	\begin{align*}
		\|f\|_{\beta;[s,t]}:=\sup_{s\le u<v\le t}\frac{\|f_v-f_u\|_U}{|u-v|^\beta}.
	\end{align*}
	If $\|f\|_{\beta;[0,T]}<+\infty$, then we have $f\in C^{{1/\beta}\text{-}var}(U;[0,T])$. Moreover, if $g$ also satisfies $\|g\|_{\beta;[0,T]}<+\infty$, then the $1/\beta$-variation of the function $fg:(r_1,r_2)\mapsto f_{r_1}g_{r_2}$ 
	defined on $[0,T]^2$ is finite.
\end{remark}

\begin{definition}
	Assume $f\in C^{p\text{-}var}(U,[0,T]^2)$ and $g\in C^{q\text{-}var}(W,[0,T]^2)$. If $$\frac1p+\frac1q>1,$$ then we say that  
	$f$ and $g$ have complementary regularity.
\end{definition}

\begin{lemma}(\cite{FV20111,To2002})\label{Young}
	Given $f:[0,T]^2\rightarrow \mathcal{L}(U,W)$ and $g:[0,T]^2\rightarrow U$. Then the following $2$D Young integral is defined as
	\begin{align*}
		\int_{[0,T]^2}f_{r_1,r_2}{\rm d}g_{r_1,r_2}:=\lim_{|\pi|\rightarrow 0}\sum_{i,j}f_{u_i,v_j}		g([u_i,u_{i+1}]\times[v_j,v_{j+1}])
	\end{align*}
	if the limit above exists.
	Moreover, if $f$ and $g$ have complementary regularity, then 
	$\int_{[0,T]^2}f_{r_1,r_2}{\rm d}g_{r_1,r_2}$ exists, and it holds that 
	\begin{align*}
		\bigg\|\int_{[0,T]^2}f_{r_1,r_2}{\rm d}g_{r_1,r_2}\bigg\|_W
		\le C_{p,q}|||f|||_{V^p;[0,T]^2}\|g\|_{V^q;[0,T]^2},
	\end{align*}
	where 
	\begin{align*}
		|||f|||_{V^p;[0,T]^2}:=\|f_{0,0}\|_{\mathcal{L}(U,W)}+\|f_{0,\cdot}\|_{p\text{-}var;[0,T]}+\|f_{\cdot,0}\|_{p\text{-}var;[0,T]}+\|f\|_{V^p;[0,T]^2}.
	\end{align*}
	In particular,  the result can also be  restricted to $[s,t]\times[u,v]\subseteq[0,T]^2$. 
\end{lemma}

\subsection{Malliavin calculus}
Let $(\Omega,\mathcal{F},\mathbb{P})$  be a probability space. 
\begin{definition}\label{fBm}
	The scalar-valued fractional Brownian motion $B=\{B_t\}_{t\in [0,T]}$ is a continuous centered Gaussian process with $B_0=0$ almost surely and the covariance
	\begin{align*}
		R_{s,t}:=\mathbb{E}\big[B_s B_t\big]=\frac12\bigg(t^{2H}+s^{2H}-|t-s|^{2H}\bigg),\quad \ s,t\in [0,T].
	\end{align*}
	Here, $H\in(0,1)$ is called the Hurst parameter of $B$.
\end{definition}

Based on Definition \ref{fBm}, the regularity of  the fBm, as well as the regularity of its covariance, is obtained.

\begin{lemma}\cite[Chapter 5]{Nualart}\label{R}
	For $H\in(0,1)$ and $p\ge 1$, there exists a constant $C=C_{p}$ such that 
	\begin{align*}
		\sup_{0\le s<t\le T}\frac{\|B_t-B_s\|_{L^p(\Omega)}}{|t-s|^H}\le C.
	\end{align*}
	Meanwhile, for any $\beta\in(0,H)$, there exists a nonnegative random variable $G=G_{\beta,T}\in L^p(\Omega)$ for all $p\ge 1$,  such that 
	$\|B\|_{\beta;[0,T]}\le G$ almost surely.
\end{lemma}

\begin{lemma}\cite[Example 1]{FV20111}\label{RR}
	For $H\in (0,\frac12]$, we have 
	\begin{align*}
		\quad R\in C^{1/2H\text{-}var}(\mathbb{R};[0,T]^2).
	\end{align*}
	More precisely, it holds that $\|R\|_{V^{1/2H};[s,t]^2}\le C_H|t-s|^{2H}$.
\end{lemma}

Combining Lemmas \ref{Young}-\ref{RR}, for a function $f:[0,T]^2\rightarrow \mathbb{R}$ sharing a similar regularity of  $B$, i.e., $f\in C^{1/\beta\text{-}var}(\mathbb{R};[0,T]^2)$ with $\beta=H^-$, we obtain that
\begin{align*}
	\int_{[0,T]^2}f_{r_1,r_2}{\rm d}R_{r_1,r_2}
\end{align*}
is well-defined as long as $H\in(\frac13,\frac12)$. 
In the following, based on the $2$D Young integral, we introduce the Malliavin calculus associated to the fBm with Hurst parameter $H\in(\frac13,\frac12)$. 

Noticing 
\begin{align*}
	R_{s,t}=\int_{[0,s]\times[0,t]}{\rm d}R_{r_1,r_2}=\int_{[0,T]^2}\mathds{1}_{[0,s]}(r_1)\mathds{1}_{[0,t]}(r_2){\rm d}R_{r_1,r_2}
\end{align*}
with $\mathds{1}_{[0,t]}(\cdot)$ being the indicator function,
we consider the inner product
\begin{align*}
	\langle\mathds{1}_{[0,t]},\mathds{1}_{[0,s]}\rangle_{\mathcal{H}}:=R_{s,t},
\end{align*}
which yields a Hilbert space $\left(\mathcal{H},\langle\cdot,\cdot\rangle_{\mathcal{H}}\right)$ being the closure of the space of all step functions on $[0,T]$ with respect to $\langle\cdot,\cdot\rangle_{\mathcal{H}}$. 
\begin{definition}
	Given a random variable 
	\begin{align*}
		F=f(B_{t_1},\cdots,B_{t_N}),
	\end{align*}
	where $t_1,\cdots,t_N\in[0,T]$, and $f:\mathbb{R}^N\rightarrow \mathbb{R}$ is a bounded smooth function with  derivatives bounded up to any order, the Malliavin derivative of $F$ is defined by 
	\begin{align*}
		D_{\cdot}F:=\sum_{i=1}^{N}\frac{\partial f}{\partial x_i}(B_{t_1},\cdots,B_{t_N})\mathds{1}_{[0,t_i]}(\cdot).
	\end{align*}
	Furthermore, for $p\ge 1$, the space $\mathbb{D}^{1,p}$ is the closure of the set of random variables in terms of the norm 
	\begin{align*}
		\|F\|_{\mathbb{D}^{1,p}}:=\Big( \mathbb{E}\big[|F|^p\big] + \mathbb{E}\big[\|DF\|_{\mathcal{H}}^p\big]  \Big)^{\frac1p}.
	\end{align*}
\end{definition}

\begin{definition}
	Given an $\mathcal{H}$-valued random variable $\varphi\in L^2(\Omega;\mathcal{H})$ satisfying
	\begin{align*}
		\Big| \mathbb{E}\big[\langle\varphi,DF\rangle_{\mathcal{H}}\big]  \Big|\le C_\varphi \|F\|_{L^2(\Omega)},  \quad  F\in \mathbb{D}^{1,2},
	\end{align*}
	the adjoint operator $\delta$ of the derivative operator $D$ acts on $\varphi$ is $\delta(\varphi)\in L^2(\Omega;\mathbb{R})$  such that 
	\begin{align*}
		\mathbb{E}\big[\langle\varphi,DF\rangle_{\mathcal{H}}\big]=\mathbb{E}\big[F\delta(\varphi)\big]
	\end{align*}
	for all $F\in \mathbb{D}^{1,2}$. In this case, we say $\varphi\in {\rm Dom}(\delta)$.
	Furthermore, 
	the Skorohod integral of $\varphi$ with respect to $B$ is defined by 
	\begin{align*}
		\int_{0}^{T}\varphi_t \delta B_t:=\delta(\varphi).
	\end{align*}
	In particular, 
	for $t\in [0,T]$, $\int_{0}^{t}\varphi_u \delta B_u:=\delta\big(\varphi \mathds{1}_{[0,t]}\big)$. 
\end{definition}

%
%

On the other hand, the fBm with Hurst parameter $H\in(\frac14,\frac12)$ can be naturally  lifted to the rough path almost surely, which leads to the rough integral  $\int_{0}^{T}\varphi_t{\rm d}B_t$  and the solutions to SDEs in the sense of rough paths  \cite{MHFriz,Lyons}. 
In the sequel, we introduce the transformation formula for the Skorohod integral and the rough integral, which is essential for us in numerical analysis.

\begin{lemma}\cite{CL20199,ST}\label{S}
	Let $H\in(\frac13,\frac12)$. Assume that the first order derivative of the function $\Phi$ is bounded and that $\varphi$ solves
	\begin{align*}
		{\rm d}\varphi_t&=\Phi(\varphi_t) {\rm d}B_t
	\end{align*}
	in the sense of  rough path. Then it holds  almost surely that
	\begin{align*}
		\int_{0}^{T}\varphi_t{\rm d}B_t=\int_{0}^{T}\varphi_t\delta B_t+H\int_{0}^{T}\Phi(\varphi_s)s^{2H-1}{\rm d}s
		+\int_{[0,T]^2} \mathds{1}_{[0,r_2]}(r_1)\Big[D_{r_1}\varphi_{r_2}-\Phi(\varphi_{r_2})\Big]{\rm d}R_{r_1,r_2}.
	\end{align*}
\end{lemma}

\section{Convergence analysis on the Euler scheme}\label{sec-3}

In this section, we set $h=\frac{T}{n}$ and $t_k=kh$, $k=0,\cdots,n$. For $t\in(t_k,t_{k+1}]$, define $\lfloor t\rfloor:=t_k$ and $\lceil t \rceil:=t_{k+1}$. Before proving the main results, we give lemmas for the solution of  \eqref{sde} and the covariance of the fBm.

\begin{lemma}\label{lm-3}
	Assume that the derivative of $a$ is bounded. Then \eqref{sde} admits a unique solution satisfying
	\begin{align*}
		\mathbb{E}\bigg[\sup_{\tau\in[0,T]}|X_{\tau}|^p\bigg]+\mathbb{E}\Big[\|X\|^p_{\beta;[0,T]}\Big]\le C, \quad p\ge 1,~\beta<H.
	\end{align*}
\end{lemma}
\begin{proof}
	Since $a$ has bounded derivative, the existence and uniqueness of the solution to \eqref{sde} is deduced from a standard argument by the contractive mapping principle. Moreover, based on
	\begin{align*}
		\sup_{\tau\in[0,t]}|X_{\tau}|&\le |X_0|+\int_{0}^{t}\sup_{\tau\in[0,s]}|a(X_\tau)|{\rm d}s+\sigma \sup_{\tau\in[0,t]}\big|B_{\tau}\big|\\
		&\le |X_0|+C\int_{0}^{t}\bigg(1+\sup_{\tau\in[0,s]}|X_\tau|\bigg){\rm d}s+\sigma \sup_{\tau\in[0,T]}\big|B_{\tau}\big|,
	\end{align*}
	Gronwall's inequality gives
	\begin{align*}
		\sup_{\tau\in[0,t]}|X_{\tau}|\le C \bigg(1+\sup_{\tau\in[0,T]}\big|B_{\tau}\big|\bigg).
	\end{align*}
	Then Lemma \ref{R} yields
	\begin{align*}
		\mathbb{E}\bigg[\sup_{\tau\in[0,T]}|X_{\tau}|^p\bigg]\le C, \quad p\ge 1.
	\end{align*}
	
	On the other hand, we have
	\begin{align*}
		|X_{t}-X_s|&\le \int_{s}^{t}|a(X_\tau)|{\rm d}\tau+\sigma \big|B_{t}-B_s\big|\\
		&\le C\int_{s}^{t}\big(1+|X_\tau|\big){\rm d}\tau+\sigma \big|B_{t}-B_s\big|\\
		&\le C \bigg(1+\sup_{\tau\in[0,T]}|B_{\tau}|\bigg)|t-s|+\sigma \big|B_{t}-B_s\big|,
	\end{align*}
	which implies
	\begin{align*}
		\mathbb{E}\Big[\|X\|^p_{\beta;[0,T]}\Big]\le C,
	\end{align*}
	for any $p\ge 1$ and $\beta<H$.
\end{proof}

\begin{lemma}\label{lm-1}
	Let $R$ be the covariance of the fractional Brownian motion $B$ with Hurst parameter $H\in (0,\frac12)$. Then it holds that
	\begin{align}
		\int_{0}^{T}\int_{0}^{T}\|R\|_{V^{1/2H};[\lfloor t\rfloor,t]\times[\lfloor s\rfloor,s]}{\rm d}s{\rm d}t\le Ch^{2H+1},\label{R2}
	\end{align}
	and 
	\begin{align}
		\int_{0}^{T} \|R\|_{V^{1/2H};[\lfloor t\rfloor,t]^2}{\rm d}t
		+\int_{0}^{T} \|R\|_{V^{1/2H};[0,\lfloor t\rfloor]\times[\lfloor t\rfloor,t]}{\rm d}t
		\le Ch^{2H}.\label{R22}
	\end{align}
\end{lemma}
\begin{proof}
	We decomposite 
	\begin{align*}
		\int_{0}^{T}\int_{0}^{T}\|R\|_{V^{1/2H};[\lfloor t\rfloor,t]\times[\lfloor s\rfloor,s]}{\rm d}s{\rm d}t
	\end{align*}
	into 
	\begin{align*}
		&\int_{0}^{T}\int_{0}^{T}\|R\|_{V^{1/2H};[\lfloor t\rfloor,t]\times[\lfloor s\rfloor,s]}\mathds{1}_{[\lfloor t\rfloor,\lceil t \rceil]}(s){\rm d}s{\rm d}t\\
		+
		&\int_{0}^{T}\int_{0}^{T}\|R\|_{V^{1/2H};[\lfloor t\rfloor,t]\times[\lfloor s\rfloor,s]}\mathds{1}_{[0,T]\backslash[\lfloor t\rfloor,\lceil t \rceil]}(s){\rm d}s{\rm d}t=:I_{1}+I_{2}.
	\end{align*}
	By means of Lemma \ref{RR}, 	we get 
	\begin{align*}
		I_{1}\le \int_{0}^{T}\int_{0}^{T}h^{2H}\mathds{1}_{[\lfloor t\rfloor,\lceil t \rceil]}(s){\rm d}s{\rm d}t
		\le Ch^{2H+1}.
	\end{align*}
	For the part $I_{2}$,
	notice that if $s\notin [\lfloor t\rfloor,\lceil t \rceil]$, then the sets $[\lfloor t\rfloor,\lceil t \rceil]$ and $[\lfloor s\rfloor,\lceil s \rceil]$ are essentially disjoint. We claim that for any two essentially disjoint sets, $[a,b]$ and $[c,d]$ with $a<b\le c<d$, the covariance of the increments of the fBm is negative. Indeed, due to $H<1/2$, it holds that
	\begin{align*}
		&\mathbb{E}\Big[\big(B_b-B_a\big)\big(B_d-B_c\big)\Big]\\
		=&\frac12\Big[(d-a)^{2H}-(d-b)^{2H}+(c-b)^{2H}-(c-a)^{2H}\Big]\\
		=&H\bigg(\int_{a}^{b}(d-u)^{2H-1}{\rm d}u-\int_{a}^{b}(c-u)^{2H-1}{\rm d}u\bigg)\\
		=&H(2H-1)\bigg(\int_{a}^{b}\int_{c}^{d}(v-u)^{2H-2}{\rm d}v{\rm d}u\bigg)< 0.
	\end{align*}
	It then leads to
	\begin{align*}
		\|R\|^{1/2H}_{V^{1/2H};[\lfloor t\rfloor,t]\times[\lfloor s\rfloor,s]}&=\sup_{\pi}\sum_{i,j}\big|R([u_i,u_{i+1}]\times[v_j,v_{j+1}])\big|^{1/2H}\\
		&\le \sup_{\pi}\bigg|\sum_{i,j}R([u_i,u_{i+1}]\times[v_j,v_{j+1}])\bigg|^{1/2H}\\
		&=\bigg|\mathbb{E}\Big[\big(B_s-B_{\lfloor s\rfloor}\big)\big(B_t-B_{\lfloor t\rfloor}\big)\Big]\bigg|^{1/2H}
	\end{align*}
	with $\pi=\{(u_i,v_j)\}$ being
	a partition of $[\lfloor t\rfloor,t]\times[\lfloor s\rfloor,s]$, which yields 
	\begin{align*}
		\|R\|_{V^{1/2H};[\lfloor t\rfloor,t]\times[\lfloor s\rfloor,s]}&\le \bigg|\mathbb{E}\Big[\big(B_s-B_{\lfloor s\rfloor}\big)\big(B_t-B_{\lfloor t\rfloor}\big)\Big]\bigg|\\
		&=H(1-2H)\bigg(\int_{\lfloor t\rfloor}^{t}\int_{\lfloor s\rfloor}^{s}|v-u|^{2H-2}{\rm d}v{\rm d}u\bigg).
	\end{align*}
	Then we obtain 
	\begin{align*}
		I_{2}&\le C\int_{0}^{T}\int_{0}^{T}\int_{\lfloor t\rfloor}^{t}\int_{\lfloor s\rfloor}^{s}|v-u|^{2H-2}{\rm d}v{\rm d}u\mathds{1}_{[0,T]\backslash[\lfloor t\rfloor,\lceil t \rceil]}(s){\rm d}s{\rm d}t\\
		&=C\sum_{i=0}^{n-1}\int_{t_i}^{t_{i+1}}\bigg(\int_{0}^{t_i}+\int_{t_{i+1}}^{T}\bigg)\int_{ t_i}^{t}\int_{\lfloor s\rfloor}^{s}|v-u|^{2H-2}{\rm d}v{\rm d}u{\rm d}s{\rm d}t\\
		&= C\sum_{i=0}^{n-1}\int_{t_i}^{t_{i+1}}\int_{ u}^{t_{i+1}}\bigg(\int_{0}^{t_i}+\int_{t_{i+1}}^{T}\bigg)\int_{\lfloor s\rfloor}^{s}|v-u|^{2H-2}{\rm d}v{\rm d}s{\rm d}t{\rm d}u\\
		&=C\sum_{i=0}^{n-1}\int_{t_i}^{t_{i+1}}(t_{i+1}-u)\bigg(\int_{0}^{t_i}+\int_{t_{i+1}}^{T}\bigg)\int_{\lfloor s\rfloor}^{s}|v-u|^{2H-2}{\rm d}v{\rm d}s{\rm d}u\\
		&=C\sum_{i=0}^{n-1}\int_{t_i}^{t_{i+1}}(t_{i+1}-u)\bigg(\int_{0}^{t_i}+\int_{t_{i+1}}^{T}\bigg)\int_{v}^{\lceil v \rceil}|v-u|^{2H-2}{\rm d}s{\rm d}v{\rm d}u\\
		&=C\sum_{i=0}^{n-1}\int_{t_i}^{t_{i+1}}\bigg(\int_{0}^{t_i}+\int_{t_{i+1}}^{T}\bigg)(t_{i+1}-u)(\lceil v \rceil -v)| v-u|^{2H-2}{\rm d}v{\rm d}u\\
		&\le C h^2 \sum_{i=0}^{n-1}\int_{t_i}^{t_{i+1}}\bigg(\int_{0}^{t_i}+\int_{t_{i+1}}^{T}\bigg)| v-u|^{2H-2}{\rm d}v{\rm d}u.
	\end{align*}
	By direct calculations,  we derive
	\begin{align*}
		&\sum_{i=0}^{n-1}\int_{t_i}^{t_{i+1}}\bigg(\int_{0}^{t_i}+\int_{t_{i+1}}^{T}\bigg)| v-u|^{2H-2}{\rm d}v{\rm d}u\\
		=& \sum_{i=0}^{n-1}\int_{t_i}^{t_{i+1}}\bigg(\int_{0}^{t_i}(u-v)^{2H-2}{\rm d}v+\int_{t_{i+1}}^{T}(v-u)^{2H-2}{\rm d}v\bigg){\rm d}u\\
		=&\frac{1}{1-2H} \sum_{i=0}^{n-1}\int_{t_i}^{t_{i+1}}\bigg((u-t_i)^{2H-1}-u^{2H-1}+(t_{i+1}-u)^{2H-1}
		-(T-u)^{2H-1}\bigg){\rm d}u\\
		=&\frac{1}{1-2H}\bigg( \sum_{i=0}^{n-1}\int_{t_i}^{t_{i+1}}\Big((u-t_i)^{2H-1}+(t_{i+1}-u)
		^{2H-1}\Big){\rm d}u\bigg)\\
		&-\frac{1}{1-2H}\int_{0}^{T}\Big(u^{2H-1}
		+(T-u)^{2H-1}\Big){\rm d}u\\
		=&\frac{1}{2H(1-2H)} \Bigg[\bigg(\sum_{i=0}^{n-1}2h^{2H}\bigg)-2T^{2H}\Bigg]
		\le C h^{2H-1},
	\end{align*}
	which completes the proof of \eqref{R2}. 
	
	Similarly, we have 
	\begin{align*}
		\int_{0}^{T} \|R\|_{V^{1/2H};[0,\lfloor t\rfloor]\times[\lfloor t\rfloor,t]}{\rm d}t
		&=\int_{0}^{T}\bigg|\mathbb{E}\Big[\big(B_{t}-B_{\lfloor t \rfloor}\big)\big(B_{\lfloor t \rfloor}-B_0\big)\Big]\bigg|{\rm d}t\\
		&=C\int_{0}^{T}\int_{\lfloor t\rfloor}^{t}\int_{0}^{\lfloor t\rfloor}|v-u|^{2H-2}{\rm d}v{\rm d}u{\rm d}t\\
		&=C\sum_{i=0}^{n-1}\int_{t_i}^{t_{i+1}}\int_{t_i}^{t}\int_{0}^{t_i}|v-u|^{2H-2}{\rm d}v{\rm d}u{\rm d}t\\
		&=C\sum_{i=0}^{n-1}\int_{t_i}^{t_{i+1}}\int_{t}^{t_{i+1}}\int_{0}^{t_i}|v-u|^{2H-2}{\rm d}v{\rm d}t{\rm d}u\\
		&\le Ch\sum_{i=0}^{n-1}\int_{t_i}^{t_{i+1}}\int_{0}^{t_i}(u-v)^{2H-2}{\rm d}v{\rm d}u\\
		&=Ch\sum_{i=0}^{n-1}\int_{t_i}^{t_{i+1}}\frac{1}{1-2H}\Big[(u-t_i)^{2H-1}-u^{2H-1}\Big]{\rm d}u
		\le Ch^{2H}.
	\end{align*}
	Combining with 
	\begin{align*}
		\int_{0}^{T} \|R\|_{V^{1/2H};[\lfloor t\rfloor,t]^2}{\rm d}t \le Ch^{2H}
	\end{align*}
	implied by Lemma \ref{RR}, 
	the inequality \eqref{R22} is obtained.
\end{proof}

Now we are in position to prove  Theorem \ref{main}.
\begin{proof}
	By \eqref{sde}-\eqref{sch}, we have 
	\begin{align*}
		X_t-Y_t&=\int_{0}^{t}a(X_s){\rm d}s -\int_{0}^{t}a(Y_{\lfloor s \rfloor}){\rm d}s\\
		&=\int_{0}^{t}\big(a(X_{\lfloor s \rfloor})-a(Y_{\lfloor s \rfloor})\big){\rm d}s +\int_{0}^{t}\big(a(X_{s})-a(X_{\lfloor s \rfloor})\big){\rm d}s,
	\end{align*}
	which satisfies 
	\begin{align*}
		\mathbb{E}\big|X_t-Y_t\big|^2&\le C\int_{0}^{t}\mathbb{E}\big|a(X_{\lfloor s \rfloor})-a(Y_{\lfloor s \rfloor})\big|^2{\rm d}s+ C \mathbb{E}\bigg|\int_{0}^{t}\big(a(X_{s})-a(X_{\lfloor s \rfloor})\big){\rm d}s\bigg|^2.
	\end{align*}
	Taking the supremum with respect to the time variable and using the Lipschitz continuity of $a$, we have 
	\begin{align*}
		\sup_{\tau\in[0,t]}\mathbb{E}\big|X_\tau-Y_\tau\big|^2\le C\int_{0}^{t} \sup_{\tau\in[0,s]}\mathbb{E}\big|X_{\tau}-Y_{\tau}\big|^2{\rm d}s
		+ C\sup_{t\in[0,T]}  \mathbb{E}\bigg|\int_{0}^{t}\big(a(X_{s})-a(X_{\lfloor s \rfloor})\big){\rm d}s\bigg|^2.
	\end{align*}
	Together with Gronwall's inequality, in order to prove 
	\begin{align*}
		\bigg(\sup_{t\in[0,T]}\mathbb{E}\big|X_t-Y_t\big|^2\bigg)^{1/2}\le Ch^{2H},
	\end{align*}
	it suffices to show
	\begin{align}\label{lm}
		\sup_{t\in[0,T]}\mathbb{E}\bigg|\int_{0}^{t}\big(a(X_{s})-a(X_{\lfloor s \rfloor})\big){\rm d}s\bigg|^2\le C h^{4H}.
	\end{align}

	In the sequel, we focus on proving	\eqref{lm}. 
	The chain rule applied to $a(X)$ implies that $a(X)$ solves the rough differential equation
	\begin{align*}
		{\rm d}a(X_{r})=a'(X_r)a(X_r){\rm d}r
		+\sigma a'(X_r){\rm d}B_r.
	\end{align*}
	Exploiting Lemma \ref{S}, we obtain
	\begin{align*}
		&a(X_{s})-a(X_{\lfloor s \rfloor})\\
		=&\int_{\lfloor s \rfloor}^{s}a'(X_r)a(X_r){\rm d}r+\sigma\int_{\lfloor s \rfloor}^{s} a'(X_r){\rm d} B_r\\
		=&\int_{\lfloor s \rfloor}^{s}a'(X_r)a(X_r){\rm d}r+\sigma\int_{\lfloor s \rfloor}^{s}a'(X_r)\delta B_r+ \sigma^2 H\int_{\lfloor s \rfloor}^{s} a''(X_r)r^{2H-1}{\rm d}r\\
		&+\sigma\int_{[0,T]^2}\mathds{1}_{[\lfloor s \rfloor,s]}(r_2)\mathds{1}_{[0,r_2]}(r_1)
		\Big[D_{r_1}\big[a'(X_{r_2})\big]-\sigma a''(X_{r_2})\Big]{\rm d}R_{r_1,r_2}\\
		=&:J_1(s)+J_2(s)+J_3(s)+J_4(s).
	\end{align*}
	%
	It follows that 
	\begin{align*}
		&\mathbb{E}\bigg|\int_{0}^{u}\big(a(X_{s})-a(X_{\lfloor s \rfloor})\big){\rm d}s\bigg|^2\\
		=& \mathbb{E}\Bigg[\bigg(\int_{0}^{u}\big(J_1(t)+J_2(t)+J_3(t)+J_4(t)\big){\rm d}t\bigg)\bigg(\int_{0}^{u}\big(J_1(s)+J_2(s)+J_3(s)+J_4(s)\big){\rm d}s\bigg)\Bigg]\\
		=&\sum_{i,j=1}^{4} \mathbb{E}\Bigg[\bigg(\int_{0}^{u}J_i(t){\rm d}t\bigg)\bigg(\int_{0}^{u}J_j(s){\rm d}s\bigg)\Bigg]\\
		\le &\sum_{i,j=1}^{4} \Bigg(\mathbb{E}\Bigg[\bigg(\int_{0}^{u}J_i(t){\rm d}t\bigg)^2\Bigg]\Bigg)^{1/2}
		\Bigg(\mathbb{E}\Bigg[\bigg(\int_{0}^{u}J_j(t){\rm d}t\bigg)^2\Bigg]\Bigg)^{1/2}.
	\end{align*}
	It then remains to estimate $\mathbb{E}\bigg[\Big(\int_{0}^{u}J_i(t){\rm d}t\Big)^2\bigg]$ for each $i\in\{1,2,3,4\}$.
	
	For $J_1$, Lemma \ref{lm-3} leads to
	\begin{align*}
		&\mathbb{E}\Bigg[\bigg(\int_{0}^{u}J_1(t){\rm d}t\bigg)^2\Bigg]\\
		=&\mathbb{E}\bigg[\int_{0}^{u}
		\int_{\lfloor t \rfloor}^{t}a'(X_r)a(X_r){\rm d}r{\rm d}t
		\int_{0}^{u}
		\int_{\lfloor s \rfloor}^{s}a'(X_v)a(X_v){\rm d}v{\rm d}s\bigg]\\
		=&\mathbb{E}\bigg[\int_{0}^{u}
		\int_{r}^{\lceil r \rceil}a'(X_r)a(X_r){\rm d}t{\rm d}r
		\int_{0}^{u}
		\int_{v}^{\lceil v \rceil}a'(X_v)a(X_v){\rm d}s{\rm d}v\bigg]\\
		\le &h^2 \int_{0}^{u}
		\int_{0}^{u}\mathbb{E}\Big[
		\big|a'(X_r)a(X_r)a'(X_v)a(X_v)\big|\Big]{\rm d}v{\rm d}r\le Ch^2.
	\end{align*}
	
	
	For $J_2$, based on \cite[Chapter 1]{Nualart}, we have
	\begin{align*}
		&\mathbb{E}\Bigg[\bigg(\int_{0}^{u}J_2(t){\rm d}t\bigg)^2\Bigg]\\
		=&\sigma^2\mathbb{E}\bigg[\int_{0}^{u}\int_{\lfloor t\rfloor}^{t}  a'(X_r)\delta B_r{\rm d}t \int_{0}^{u}\int_{\lfloor s\rfloor}^{s}  a'(X_v)\delta B_v{\rm d}s\bigg]\\
		=&\sigma^2\int_{0}^{u}\int_{0}^{u}\mathbb{E}\bigg[\int_{\lfloor t\rfloor}^{t}  a'(X_r)\delta B_r\int_{\lfloor s\rfloor}^{s}  a'(X_v)\delta B_v\bigg]{\rm d}t{\rm d}s\\
		\le &\sigma^2\int_{0}^{u}\int_{0}^{u}\mathbb{E}\bigg[\int_{[0,T]^2} \mathds{1}_{[\lfloor t\rfloor,t]}(r_1)
		\mathds{1}_{[\lfloor s\rfloor,s]}(r_2)a'(X_{r_1}) a'(X_{r_2}){\rm d}R_{r_1,r_2}\bigg]{\rm d}t{\rm d}s\\
		&+\sigma^2\int_{0}^{u}\int_{0}^{u}\mathbb{E}\bigg[\int_{[0,T]^2}
		\int_{[0,T]^2} \mathds{1}_{[\lfloor t\rfloor,t]}(r_1)
		\mathds{1}_{[\lfloor s\rfloor,s]}(r_2)\mathds{1}_{[0,r_1]}(u_1)\mathds{1}_{[0,r_2]}(u_2)\\
		&\qquad\qquad\qquad \quad \times D_{u_1}\big[a'(X_{r_1})\big] D_{u_2}\big[a'(X_{r_2})\big]{\rm d}R_{u_1,u_2}{\rm d}R_{r_1,r_2}\bigg]{\rm d}t{\rm d}s\\
		=&:\sigma^2A_1+\sigma^2A_2.
	\end{align*}
	According to the regularity of $R$ and $X$ given in Lemma \ref{RR}and Lemma \ref{lm-3}, we get from the fact $H>1/3$ that the functions 
	\begin{align*}
		f:[0,T]^2&\rightarrow \mathbb{R},\\
		(r_1,r_2)&\mapsto f_{r_1,r_2}:=a'(X_{r_1})a'(X_{r_2})
	\end{align*}
	and $R$ have complementary regularity almost surely. Moreover, for any $p\ge 1$ and $\beta<H$, it holds that 
	\begin{align*}
		\mathbb{E}\Big[|||f|||^p_{V^{1/\beta};[0,T]^2}\Big]\le C.
	\end{align*}
	Then Lemma \ref{Young} and Lemma \ref{lm-1} produce
	\begin{align*}
		|A_{1}|
		\le C\int_{0}^{u}\int_{0}^{u} \|R\|_{V^{1/2H};[\lfloor t\rfloor,t]\times[\lfloor s\rfloor,s]}{\rm d}s{\rm d}t\le Ch^{2H+1}.
	\end{align*}
	Meanwhile, the Malliavin derivative satisfies
	\begin{align*}
		D_{u_1}\big[a'(X_{r_1})\big]=a''(X_{r_1})D_{u_1}X_{r_1}=\sigma\mathcal{J}_{r_1}
		\mathcal{J}^{-1}_{u_1}a''(X_{u_1}),
	\end{align*}
	where $\mathcal{J}$ and $\mathcal{J}^{-1}$ solve the linear system \cite{CL20199}
	\begin{align*}
		\left\{
		\begin{aligned}
			\mathcal{J}_t&=1+\int_{0}^{t}a'(X_s)\mathcal{J}_s{\rm d}s,\\
			\mathcal{J}^{-1}_t&=1+\int_{0}^{t}\mathcal{J}^{-1}_sa'(X_s){\rm d}s.
		\end{aligned}
		\right.
	\end{align*}
	Since the second order derivative of $a$ is bounded, it implies that  the functions
	\begin{align*}
		\tilde{f}:[0,T]^2&\rightarrow \mathbb{R},\\
		(r_1,r_2)&\mapsto \tilde{f}_{r_1,r_2}:=\int_{[0,T]^2}\mathds{1}_{[0,r_1]}(u_1)\mathds{1}_{[0,r_2]}(u_2) \mathbb{E}\Big[D_{u_1}\big[a'(X_{r_1})\big] D_{u_2}\big[a'(X_{r_2})\big]\Big]{\rm d}R_{u_1,u_2}
	\end{align*}
	and $R$ have complementary regularity almost surely, and $\mathbb{E}\Big[|||\tilde{f}|||^p_{V^{1/\beta};[0,T]^2}\Big]\le C$.
	Then we deduce
	\begin{align*}
		|A_{2}|
		\le C\int_{0}^{u}\int_{0}^{u} \|R\|_{V^{1/2H};[\lfloor t\rfloor,t]\times[\lfloor s\rfloor,s]}{\rm d}s{\rm d}t\le Ch^{2H+1}.
	\end{align*}
	The above estimates for  $A_1$ and $A_2$ yield
	\begin{align*}
		\mathbb{E}\Bigg[\bigg(\int_{0}^{u}J_2(t){\rm d}t\bigg)^2\Bigg]\le C h^{2H+1}.
	\end{align*}
	
	For $J_3$, due to $H>1/3$ and Lemma \ref{lm-3}, it holds that 
	\begin{align*}
		\mathbb{E}\Bigg[\bigg(\int_{0}^{u}J_3(t){\rm d}t\bigg)^2\Bigg]
		&=\sigma^4H^2\mathbb{E}\bigg[\int_{0}^{u}
		\int_{\lfloor t \rfloor}^{t}a''(X_r)r^{2H-1}{\rm d}r{\rm d}t
		\int_{0}^{u}
		\int_{\lfloor s \rfloor}^{s}a''(X_v)v^{2H-1}{\rm d}v{\rm d}s\bigg]\\
		&=\sigma^4H^2\mathbb{E}\bigg[\int_{0}^{u}
		\int_{r}^{\lceil r \rceil}a''(X_r)r^{2H-1}{\rm d}t{\rm d}r
		\int_{0}^{u}
		\int_{v}^{\lceil v \rceil}a''(X_v)v^{2H-1}{\rm d}s{\rm d}v\bigg]\\
		&\le Ch^2 \int_{0}^{u}
		\int_{0}^{u}
		\mathbb{E}\Big[\big|a''(X_r)a''(X_v)\big|\Big]r^{2H-1}v^{2H-1}{\rm d}v{\rm d}r\le Ch^2.
	\end{align*}
	
	For $J_4$, recall that 
	\begin{align*}
		&\mathbb{E}\Bigg[\bigg(\int_{0}^{u}J_4(t){\rm d}t\bigg)^2\Bigg]\\
		=&\sigma^2\mathbb{E}\Bigg[\bigg(\int_{0}^{u}
		\int_{[0,T]^2}
		\mathds{1}_{[\lfloor t\rfloor,t]}(r_2)\mathds{1}_{[0,r_2]}(r_1)
		\Big[D_{r_1}\big[a'(X_{r_2})\big]-\sigma a''(X_{r_2})\Big]{\rm d}R_{r_1,r_2}{\rm d}t\bigg)\\
		&~\times\bigg(\int_{0}^{u}
		\int_{[0,T]^2}
		\mathds{1}_{[\lfloor s\rfloor,s]}(r_4)\mathds{1}_{[0,r_4]}(r_3)
		\Big[D_{r_3}\big[a'(X_{r_4})\big]-\sigma a''(X_{r_4})\Big]{\rm d}R_{r_3,r_4}{\rm d}s\bigg)\Bigg].
	\end{align*}
	Based on the boundedness of the third order derivative of $a$, \cite[Section 6]{CL20199} implies that 
	the continuous functions
	\begin{align*}
		g:[0,T]^2&\rightarrow \mathbb{R},\\
		(r_1,r_2)&\mapsto g_{r_1,r_2}:=\mathds{1}_{[0,r_2]}(r_1)\Big[D_{r_1}\big[a'(X_{r_2})\big]-\sigma a''(X_{r_2})\Big]
	\end{align*}
	and $R$ have complementary regularity almost surely, and 
	$\mathbb{E}\Big[|||g|||^p_{V^{1/\beta};[0,T]^2}\Big]\le C$.
	Using Lemma \ref{lm-1} and the formulation
	\begin{align*}
		& \int_{0}^{u}\int_{[0,T]^2}
		\mathds{1}_{[\lfloor t\rfloor,t]}(r_2)\mathds{1}_{[0,r_2]}(r_1)
		\Big[D_{r_1}\big[a'(X_{r_2})\big]-\sigma a''(X_{r_2})\Big]{\rm d}R_{r_1,r_2}{\rm d}t \\
		=&\int_{0}^{u}\int_{[0,T]^2}
		\mathds{1}_{[\lfloor t\rfloor,t]}(r_2)
		\mathds{1}_{[0,\lfloor t\rfloor]}(r_1)g_{r_1,r_2}{\rm d}R_{r_1,r_2}{\rm d}t
		+\int_{0}^{u}\int_{[0,T]^2}
		\mathds{1}_{[\lfloor t\rfloor,t]}(r_2)
		\mathds{1}_{[\lfloor t\rfloor,t]}(r_1)
		g_{r_1,r_2}{\rm d}R_{r_1,r_2}{\rm d}t,
	\end{align*}
	we have
	\begin{align*}
		\mathbb{E}\Bigg[\bigg(\int_{0}^{u}J_4(t){\rm d}t\bigg)^2\Bigg]\le Ch^{4H},
	\end{align*}
	which completes the proof.
\end{proof}

\begin{corollary}\label{cor}
	Let $H\in(\frac13,\frac12)$. Assume that $a(x)=Ax$ with a constant $A$. Then it holds that 
	\begin{align*}
		\bigg(\sup_{t\in[0,T]}\mathbb{E}\big|X_t-Y_t\big|^2\bigg)^{1/2}\le Ch^{H+\frac12},
	\end{align*}
	where $X$ solves \eqref{sde} and $Y$ is given by the Euler scheme \eqref{sch}.
\end{corollary}
\begin{proof}
	Since $a(x)=Ax$, the second derivative of $a$ vanishes. Repeating the proof of Theorem \ref{main}, we have 
	\begin{align*}
		\mathbb{E}\Bigg[\bigg(\int_{0}^{u}J_1(t){\rm d}t\bigg)^2\Bigg]+\mathbb{E}\Bigg[\bigg(\int_{0}^{u}J_2(t){\rm d}t\bigg)^2\Bigg]\le Ch^{2H+1}
	\end{align*}
	and 	$J_3=J_4=0$. Then the result is obtained.
\end{proof}

\begin{remark}
	In the case of $H>\frac12$, the framework for Malliavin calculus  holds with $R$ being more regular and we refer to \cite{DX2021,hhkw,HHW,HuEuler,MC11} and references therein for the analysis on numerical schemes.
	In the case of $H\le\frac13$, more efforts should be paid to establish a conversion formula between the Skorohod integral and the rough integral.
	If  $H\le \frac14$, the well-posedness of SDEs with multiplicative noises in multi-dimensional cases is still an open problem. 
\end{remark}

\section{Numerical simulations}\label{sec-4}

In this section, we give numerical simulations for the SDE
\begin{align*}
	{\rm d}X_t=a(X_t){\rm d} t +\sigma {\rm d}B_t,\quad t\in(0,1]
\end{align*}
with $X_0=1$ and $B$ being the fBm with Hurst parameter $H\in(\frac13,\frac12)$.

\begin{example}\label{ex-1}
	Let the diffusion coefficient $\sigma=1$ and the drift coefficient 
	\begin{align*}
		a(x)=
		\left\{
		\begin{aligned}
			&\ln |x|,\quad |x|\ge 1,\\
			&\frac16 x^6 -\frac34 x^4 +\frac32 x^2 -\frac{11}{12},\quad |x|<1.
		\end{aligned}
		\right.
	\end{align*}
	Since $a$ has bounded derivatives up to order three, Theorem \ref{main} leads to that the mean square convergence rate of the Euler scheme \eqref{sch} is $2H$.
\end{example}

\begin{example}\label{ex-2}
	Let the diffusion coefficient $\sigma=9$ and the drift coefficient $a(x)=2x$.
	Due to the linearity of  $a$, it follows from Corollary \ref{cor} that the mean square convergence rate of the Euler scheme \eqref{sch} is $H+\frac12$.
\end{example}

In Figure\ref{f1} and Figure \ref{f2}, the error 
\begin{align*}
	\bigg(\sup_{t\in[0,T]}\mathbb{E}\big|X_t-Y_t\big|^2\bigg)^{1/2}
\end{align*}
is presented for Example \ref{ex-1} and Example \ref{ex-2}, respectively. We take the Hurst parameter of the fBm as $H=0.35,0.4,0.45$. The exact solution is simulated by the numerical solution with a fine time step size $h=\frac{1}{2^{15}}$ and the expectation is approximated by $1000$ sample paths. The numerical results support our theoretical analysis.

\begin{figure}
	\centering
	\subfigure[$H=0.35$]{
		\begin{minipage}[t]{0.3\linewidth}
			\includegraphics[height=3.6cm,width=3.6cm]{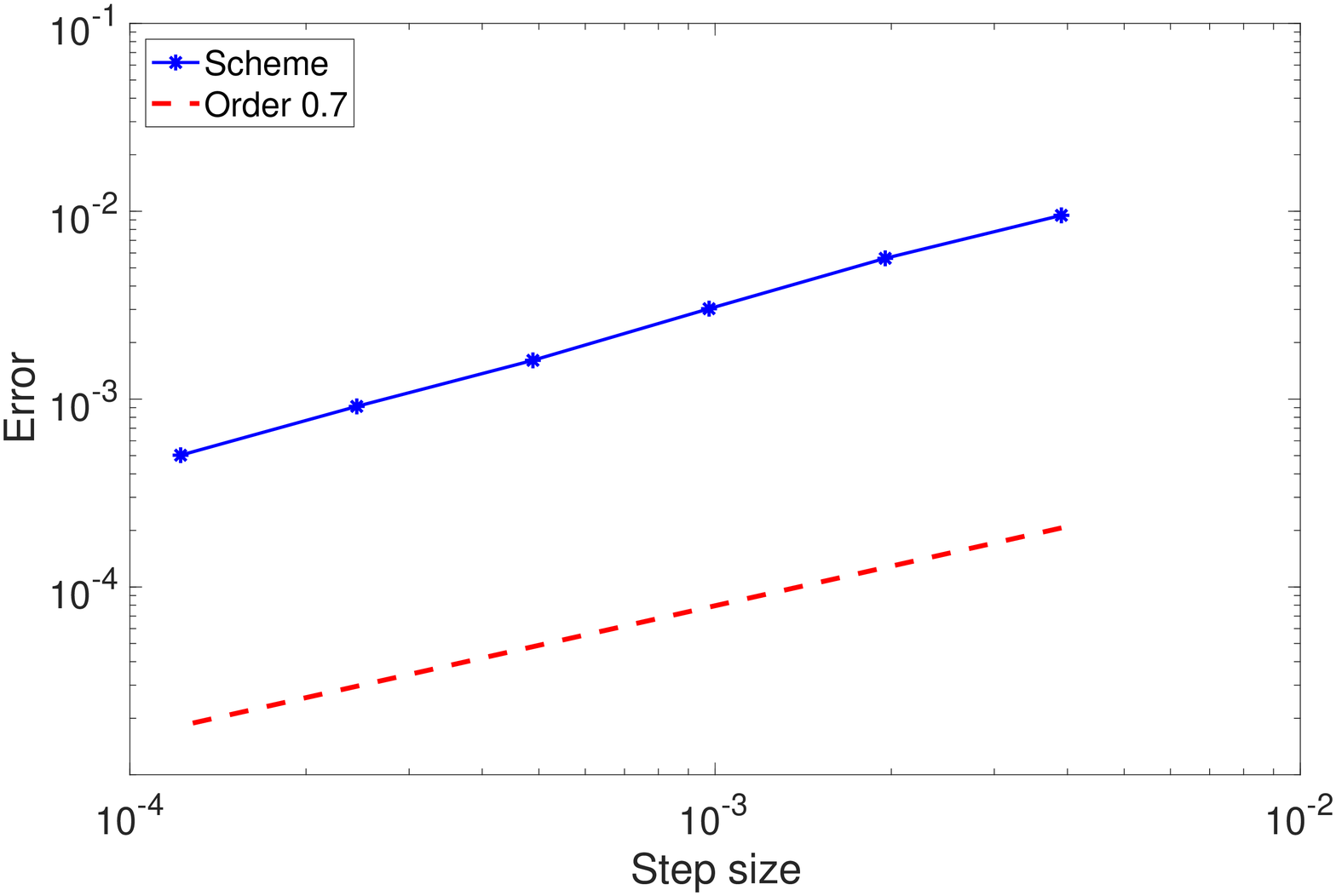}
		\end{minipage}
	}
	\subfigure[$H=0.4$]{
		\begin{minipage}[t]{0.3\linewidth}
			\includegraphics[height=3.6cm,width=3.6cm]{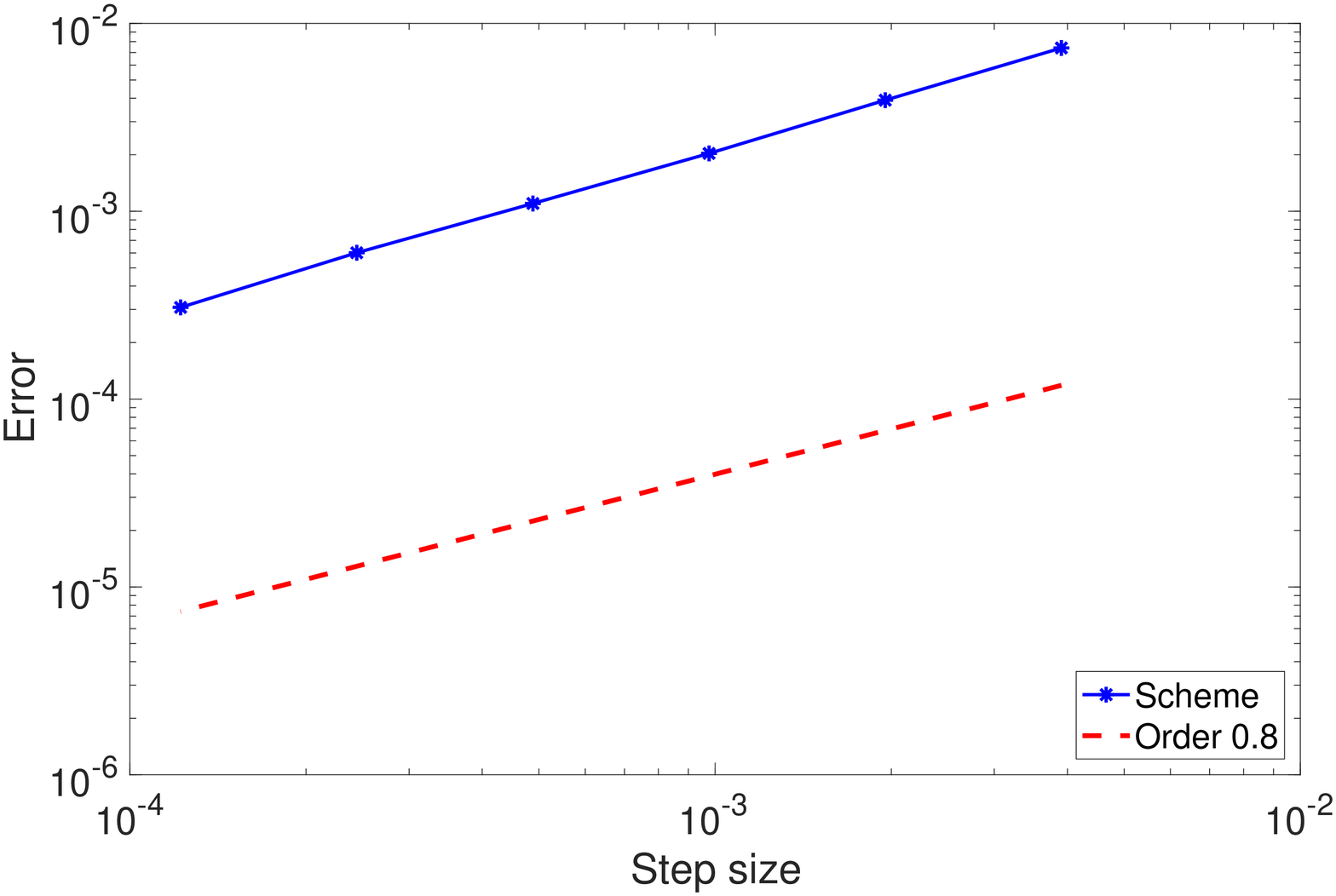}
		\end{minipage}
	}
	\subfigure[$H=0.45$]{
		\begin{minipage}[t]{0.3\linewidth}
			\includegraphics[height=3.6cm,width=3.6cm]{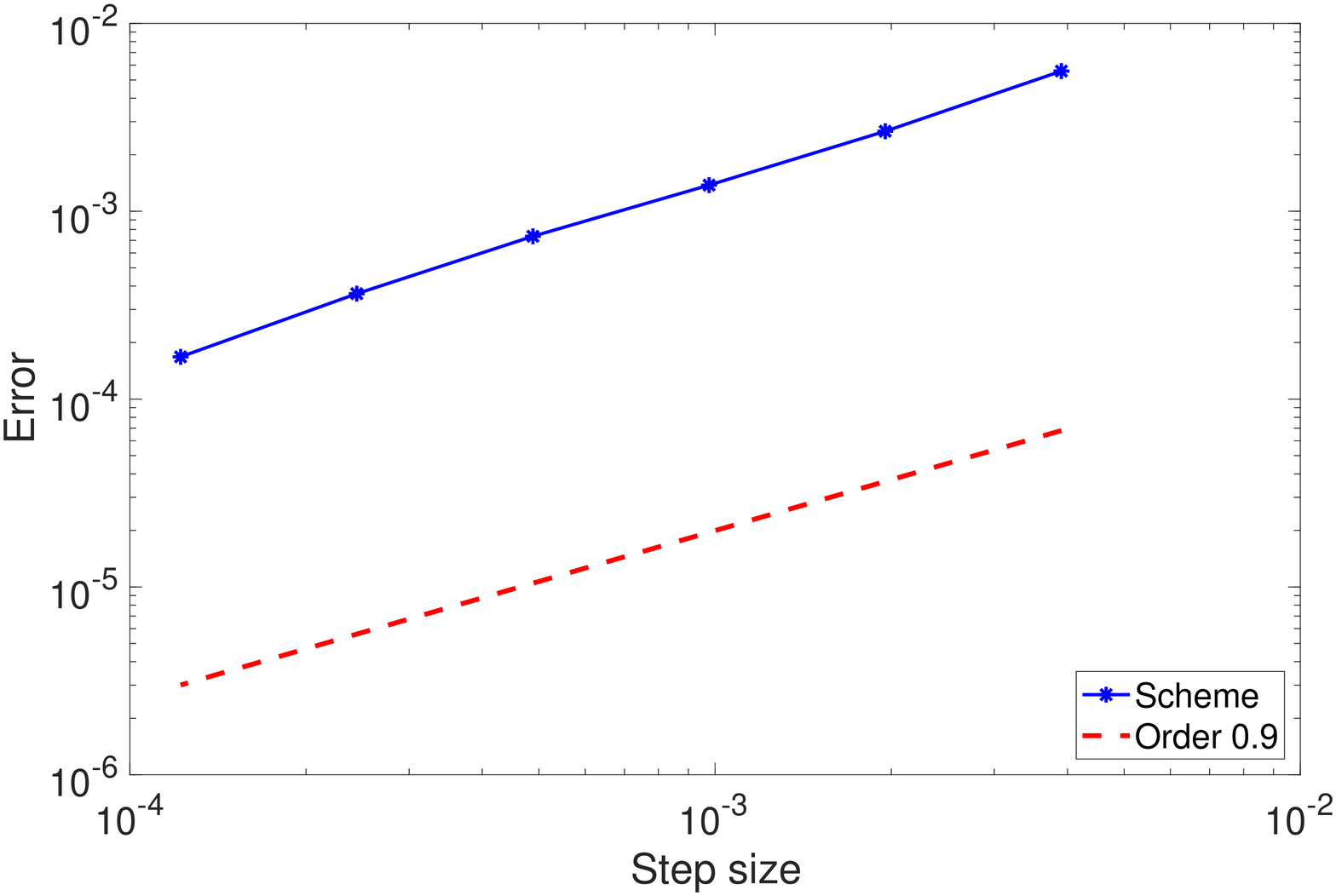}
		\end{minipage}
	}
	\caption{Error vs. Step size for Example \ref{ex-1}.}\label{f1}
\end{figure}

\begin{figure}
	\centering
	\subfigure[$H=0.35$]{
		\begin{minipage}[t]{0.3\linewidth}
			\includegraphics[height=3.6cm,width=3.6cm]{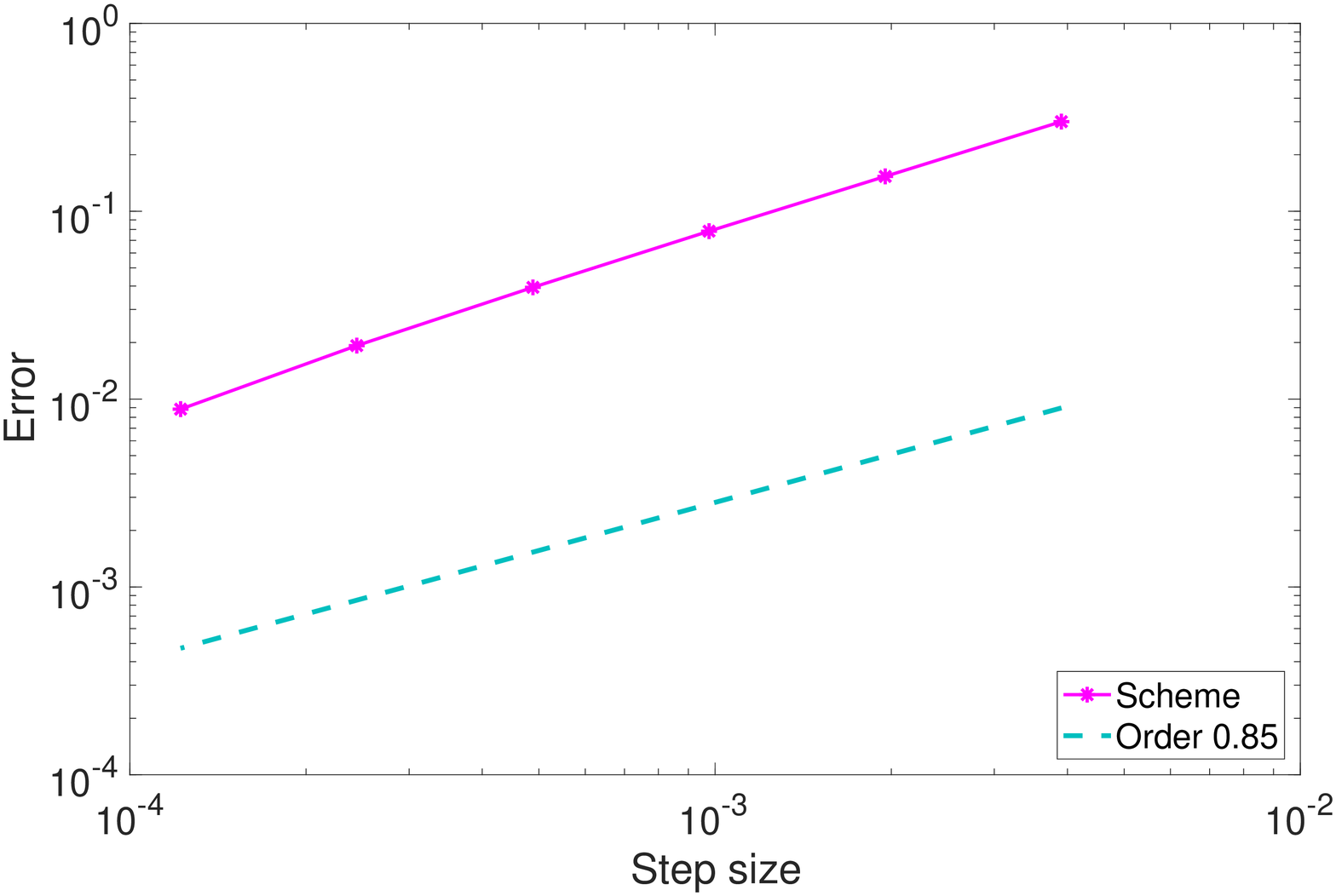}
		\end{minipage}
	}
	\subfigure[$H=0.4$]{
		\begin{minipage}[t]{0.3\linewidth}
			\includegraphics[height=3.6cm,width=3.6cm]{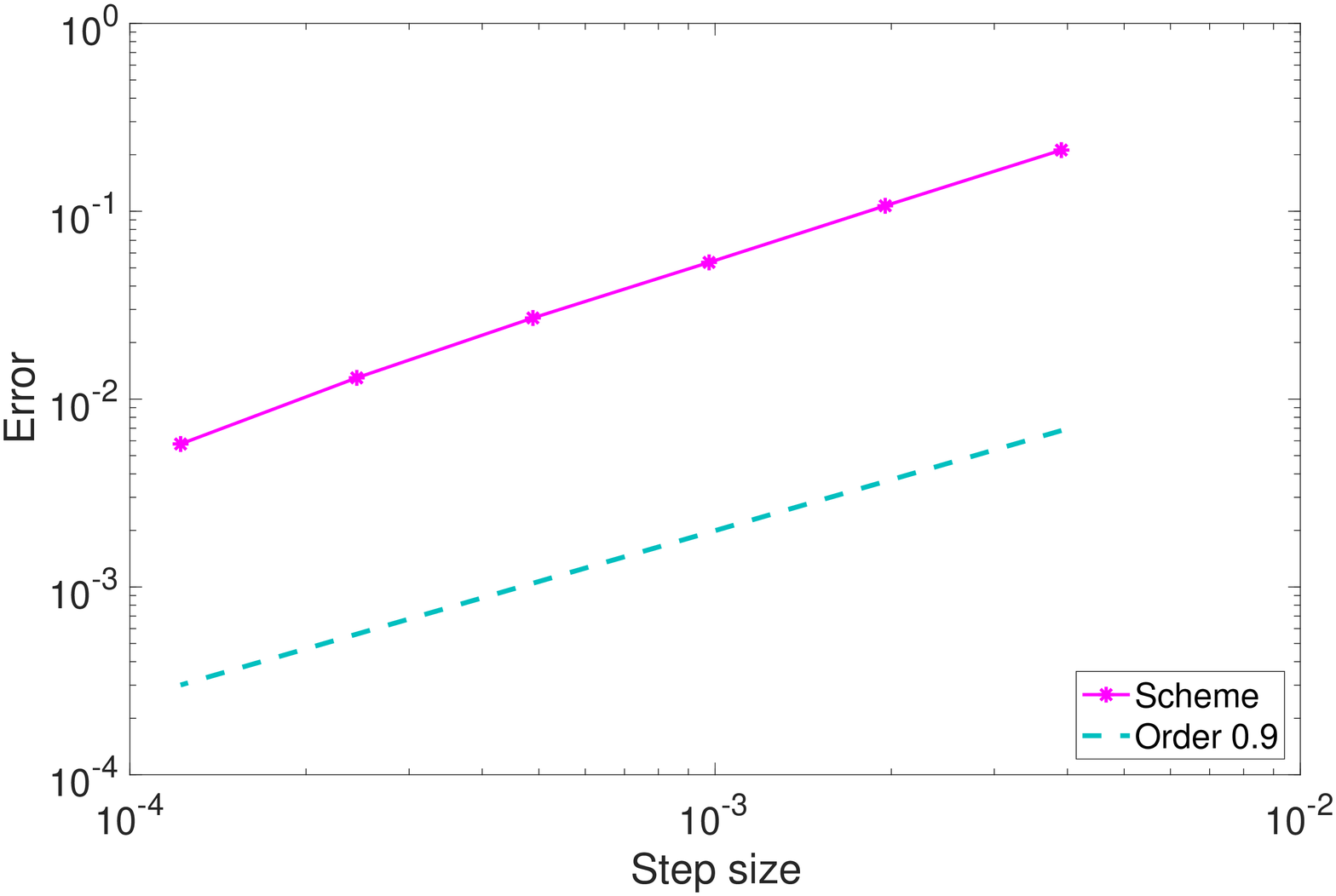}
		\end{minipage}
	}
	\subfigure[$H=0.45$]{
		\begin{minipage}[t]{0.3\linewidth}
			\includegraphics[height=3.6cm,width=3.6cm]{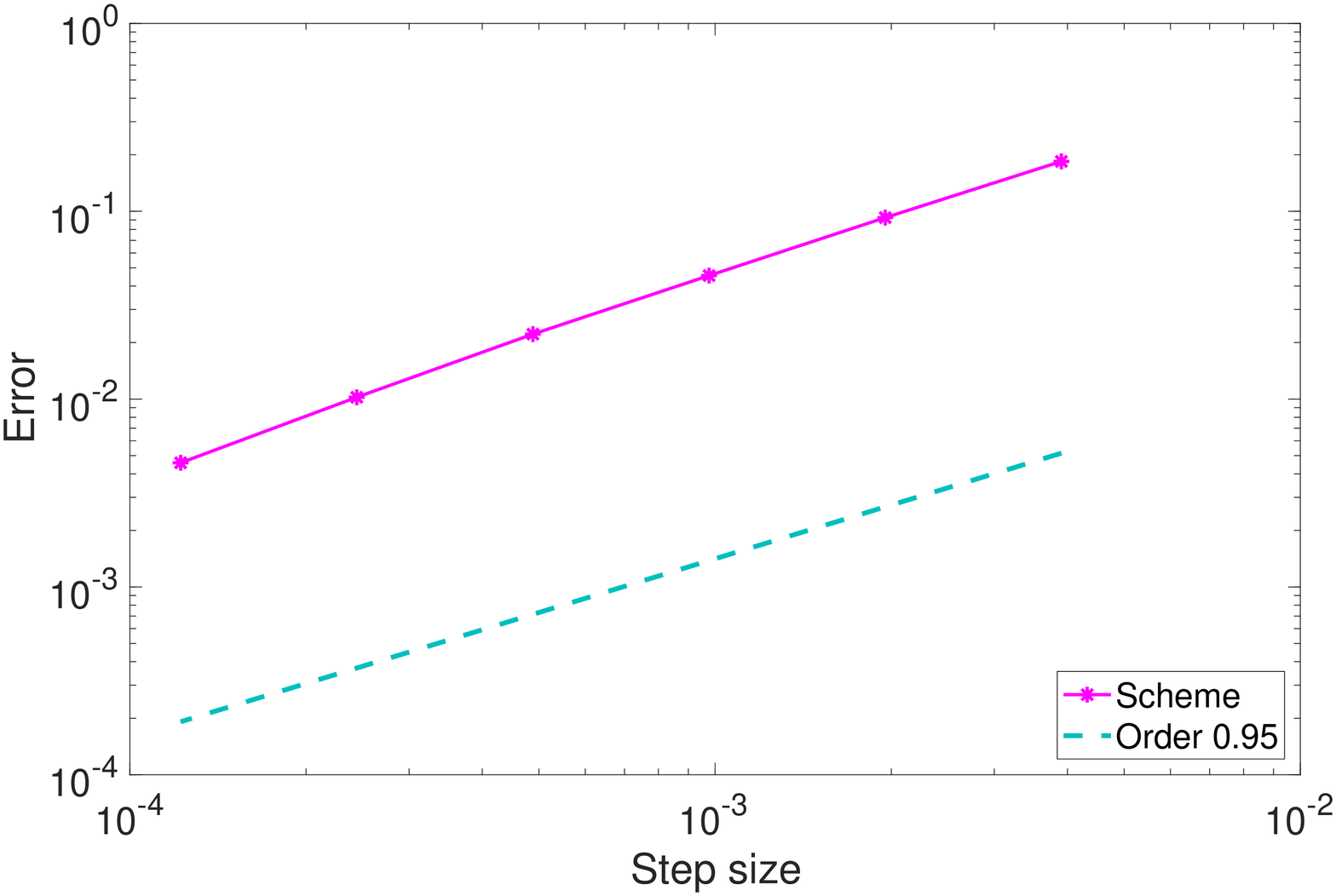}
		\end{minipage}
	}
	\caption{Error vs. Step size for Example \ref{ex-2}.}\label{f2}
\end{figure}

\section*{Acknowledgements}
The first author is funded by National Natural Science Foundation of China (NO. 12101127) and `Young and Middle-aged Teacher Education Research Project' of Fujian Province (NO. JAT200075). The second author is supported by National Natural Science Foundation of China (NO. 11971470, NO. 11871068 and NO. 12031020) and the National Key R \& D Program of China (NO. 2020YFA0713701). 

\bibliographystyle{plain}
\bibliography{mybibfile}

\end{document}